\documentclass[11pt,reqno]{amsart}
\usepackage{comment}
\usepackage{amssymb}
\usepackage{tikz-cd}
\usepackage{tikz}
\usepackage{amsthm}
\usepackage{graphicx}
\usepackage{caption}
\usepackage{subcaption}
\usepackage[backref]{hyperref}
\usepackage{cleveref} % Required for inserting images
\usepackage{booktabs}
\usepackage{float}
\usepackage{tikz-cd}
\usepackage{txfonts}
\usetikzlibrary{cd}
\usetikzlibrary{matrix, arrows}

\usepackage[disable]{todonotes}
\DeclareMathOperator{\rk}{rk}

\DeclareMathOperator{\Gal}{Gal}

\usepackage{aliascnt}
\newtheorem{theorem}{Theorem}[section]
\newcommand{\sharedtheorem}[2]{\newaliascnt{#1}{theorem}\newtheorem{#1}[#1]{#2}\aliascntresetthe{#1}}
\newtheorem*{theorem*}{Theorem}
\sharedtheorem{lemma}{Lemma}
\sharedtheorem{conjecture}{Conjecture}
\sharedtheorem{proposition}{Proposition}
\newtheorem*{proposition*}{Proposition}
\sharedtheorem{corollary}{Corollary}
\newtheorem*{corollary*}{Corollary}
\sharedtheorem{claim}{Claim}
\sharedtheorem{question}{Question}
\theoremstyle{definition}
\sharedtheorem{remark}{Remark}
\sharedtheorem{definition}{Definition}
\sharedtheorem{example}{Example}
\sharedtheorem{exercise}{Exercise}
\crefname{conjecture}{conjecture}{conjectures}\Crefname{conjecture}{Conjecture}{Conjectures}
\crefname{claim}{claim}{claims}\Crefname{claim}{Claim}{Claims}
\crefname{question}{question}{questions}\Crefname{question}{Question}{Questions}
\crefname{exercise}{exercise}{exercises}\Crefname{exercise}{Exercise}{Exercises}
\numberwithin{equation}{section}

\newcommand{\Q}{\mathbb{Q}}
\newcommand{\Z}{\mathbb{Z}}

\newcommand{\PP}{\mathbb{P}}

\newcommand{\Qbar}{\overline{\Q}}

\newcommand{\fp}{\mathfrak{p}}

\newcommand{\OO}{\mathcal{O}}

\newcommand{\tor}{\mathrm{tors}}

\newcommand{\diamondop}[1]{\langle #1 \rangle} % diamond operator

\newcommand{\githubbare}[1]{\href{https://github.com/F-Najman/q100/blob/main/#1}{\path{#1}}}

\usepackage{graphicx} % Required for inserting images

\title{Quadratic points on modular curves $X_0(N)$ for $N\leq 100$}
\author[Najman]{Filip Najman}
\address{Filip Najman, University of Zagreb, Faculty of Science, Department of Mathematics, Bijeni\v{c}ka Cesta 30, 10000 Zagreb, Croatia}
\email{\url{fnajman@math.hr}}
\author[Novak]{Ivan Novak}
\address{Ivan Novak, University of Zagreb, Faculty of Science, Department of Mathematics, Bijeni\v{c}ka Cesta 30, 10000 Zagreb, Croatia}
\email{\url{ivan.novak@math.hr}}
\date{}
\keywords{Modular curves, torsion, elliptic curves}
\begin{document}

\begin{abstract}
    We determine the quadratic points on the modular curves $X_0(N)$ for $N\leq 100$ for which this has not been previously done, namely the cases
    $$N\in\{66,70,78,82,84,86,87,88,90,93,96,99\}.$$
    We accomplish this by improving on the ``going down method," which uses the fact that we have a moduli description of all the (infinitely many) quadratic points on $X_0(n)$ for some divisor $n$ of $N$.
\end{abstract}
\thanks{The authors were financed by the Croatian
Science Foundation under the project no. IP-2022-10-5008 and by the  project “Implementation of cutting-edge research and its application as part of the Scientific Center of Excellence for Quantum and Complex Systems, and Representations of Lie Algebras“, PK.1.1.10.0004, European Union, European Regional Development Fund.}

\maketitle

\section{Introduction}
A fundamental problem in arithmetic geometry, with many consequences in the theory of elliptic curves and with many applications to Diophantine equations, is to understand the low degree points on modular curves. Perhaps the most important and well-known families of modular curves are $X_0(N)$ and $X_1(N)$. There has been notably more success in studying $X_1(N)$. In particular, the set of all $N$ for which $X_1(N)$ has degree $d$ non-cuspidal points is known for $d\leq 4$; see \cite{mazur77} for the proof for $d=1$, \cite{KM88,kamienny92} for $d=2$, \cite{Deg3Class} for $d=3$ and \cite{DerickxNajmanDeg4} for $d=4$.

On the other hand, the modular curves $X_0(N)$ turn out to be more difficult. The list of all the $N$ with non-cuspidal degree $d$ points is known only in degree $d=1$, due to Mazur \cite{mazur78} who solved the problem for prime $N$, and Kenku (see \cite{kenku1981} and the references therein) who completed the classification for all $N$.  We note that while the list of all $N$ with a non-cuspidal point on $X_0(N)$ is not known for the set of all quadratic fields, it is known for many \textit{fixed} quadratic fields (see \cite{BanwaitNajmanPadurariu24}), assuming the Generalized Riemann Hypothesis.

In this paper, we study quadratic points on $X_0(N)$. One major difficulty in studying quadratic points on $X_0(N)$ is that, as opposed to $X_1(N)$, the modular curve $X_0(N)$ has non-cuspidal quadratic points for infinitely many $N$, coming from elliptic curves with complex multiplication (CM). The set of $N$ with $X_0(N)$ having non-cuspidal points will hence be infinite. Futhermore, for a particular fixed modular curve, it is usually much easier to prove the non-existence of (non-cuspidal) rational or quadratic points, than to prove that we have found them all when some points do exist.

However, if one ignores the CM points, then it is expected that there should be only finitely many $N$ such that $X_0(N)$ has non-cuspidal non-CM points. The following conjecture is widely believed. 

\begin{conjecture}[Quadratic isogenies conjecture \cite{AKNOV24}] \label{conj:qic}
There exists an integer $C$ such that if $K$ is a quadratic field and $N>C$ is an integer, then any $P\in X_0(N)(K)$ is either a cusp or a CM point. 
\end{conjecture}

 We emphasize that the constant $C$ in this conjecture does not depend on the quadratic field $K$. \Cref{conj:qic} can be considered a special case of a conjecture of Elkies on the finiteness of the number of $N$ such that $X_0^*(N)$ has a non-CM non-cuspidal degree $d$ point (see \cite[\S 3]{Elkies2004}). To see this, note that a quadratic non-CM non-cuspidal point on $X_0(N)$ will map to a rational or quadratic non-CM non-cuspidal point on $X_0^*(N)$.

This is a topic that has received considerable attention recently. Bruin and Najman described all the quadratic points on hyperelliptic $X_0(N)$ such that the Jacobian $J_0(N)$ has rank $0$ over $\Q$ \cite{BruinNajman15}.  \"Ozman and Siksek then determined all the quadratic points on the non-hyperelliptic $X_0(N)$ with $J_0(N)$ having rank $0$ over $\Q$ and with genus $\leq 5$  \cite{OzmanSiksek19}. Box determined the quadratic points on all the $X_0(N)$ with  genus $\leq 5$ and with $J_0(N)$ having positive rank over $\Q$ \cite{Box19}. Najman and Vukorepa described all the quadratic points on the bielliptic $X_0(N)$ (for which this had not already been done in previous work) \cite{NajmanVukorepa23}. Adžaga, Keller, Michaud-Jacobs, Najman, \"Ozman and Vukorepa developed methods to compute quadratic points on a general fixed $X_0(N)$, and found all the quadratic points on all such curves of genus $\leq 8$, and of genus $\leq 10$ for $N$ prime (see \cite{AKNOV24}). 

We make progress on quadratic points on this topic by 
%in two (connected) directions. The first direction is that 
determining the quadratic points on all $X_0(N)$ for $N\leq 100$, for which this has not been already done. While certainly some further values of $N$ could have been dealt with, it was probably not possible to get a complete classification of quadratic points for all $N\leq C$ for some $C$ much larger than $100$, as the genera of the curves involved get large and the necessary computations infeasible, so $100$ seemed a reasonable place to stop.

There are $12$ such values of $N$ that we deal with; the genus of these $X_0(N)$ is $9\leq g(X_0(N))\leq 11.$ In all but one ($N=86$) of these cases the algorithms of \cite{AKNOV24} fail, in most cases because we are unable to determine $J_0(N)(\Q)_{\tor}$. In all of the aforementioned $12$ cases, $N$ is of the form $N=nd$, where $n$ is such that the points on $X_0(n)$ have already been determined. This allows us to approach the problem using the ``going-down" approach (as referred to in \cite{AKNOV24}), which was developed in \cite[Section 2]{NajmanVukorepa23}, and reduces the problem to finding rational points on some quotients of modular curves. This approach has the advantage of being less computational compared to the Atkin-Lehner sieve from \cite{AKNOV24}, in the sense that it usually requires very little computation. The Atkin-Lehner sieve is a variant of a Mordell-Weil sieve involving an Atkin--Lehner involution on the curve $X_0(N)$, together with a symmetric Chabauty criterion. This method is based on ideas of Siksek \cite{siksek}, further developed in \cite{AKNOV24, Box19,NajmanVukorepa23}. A successful application of the sieve reduces the problem to considering fixed points of an Atkin--Lehner involution and the rational points on a given Atkin--Lehner quotient.

We note that some of our results are used in the upcoming work of Derickx and Najman \cite{DN_sporadic} which determines all $N$ such that $X_0(N)$ has a sporadic point. 

We introduce conventions and notation in \Cref{sec:notation}. The going-down method, as stated in \cite[Section 2]{NajmanVukorepa23} cannot be applied to $N=96$, so we greatly generalize the method. In particular, in \cite{NajmanVukorepa23} the $m$ in the notation of \Cref{prop:2.1}, needed to be either $4$ or a prime (and coprime to $d$). In \Cref{prop:2.1} we adapt the method to work with with $m$ being an arbitrary positive integer coprime to $d$. See \Cref{sec:going-down} for more details. We describe the calculations for all the $X_0(N)$ in \Cref{sec:results}, and finally list the results in the tables in \Cref{sec:tables}.

All the code used in the paper can be found at 

\begin{center} \url{https://github.com/F-Najman/q100}\end{center}

This version of the paper differs from the published version (Res. Number Theory 12 (2026), Article 23, DOI 10.1007/s40993-026-00713-6) in two respects. First, it includes the case $N=93$, which is missing from the published version; see \Cref{sec:93}. Second, it corrects the following misprints of the published version: in \Cref{cor:main}, $k$ is a divisor of $m_1$, not of $m_2$; in \Cref{rem:recover96}, $m=3$, not $5$; in \Cref{table86}, the point $P_2$ has CM by the order of discriminant $-12$, not $-3$; in \Cref{table88}, the multiplicity of $P_2$ is $4$, not $2$; in \Cref{table99}, the point $P_2$ has $j=-32768$, not $j=-32678$, and the multiplicity of $P_3$ is $1$, not $4$.

\section*{Acknowledgements}
We thank Nikola Adžaga and Maarten Derickx for helpful discussions, and Nikola Adžaga for pointing out that the case $N=93$ was missing from the published version of the paper. We are grateful to the anonymous referees for their numerous valuable comments that improved the exposition in the paper. 

\section{Conventions and notation}\label{sec:notation}

We will be using the conventions and notation from \cite[Section 2]{NajmanVukorepa23}, which we will essentially repeat here for the convenience of the reader. 

The modular curve $X_0(N)$ is the moduli space of pairs $(E,C)$, where $E$ is a generalized elliptic curve with a cyclic subgroup $C$. It is a coarse moduli space. 

The notion of $\Q$-curves can be defined over general number fields, but in this paper we will restrict to $\Q$-curves over quadratic fields, as this is exclusively what we will be working with. For the general theory, see \cite{CremonaNajmanQCurve,Elkies2004}. An elliptic curve $E$ over a quadratic field $K$ is called a $\Q$-curve if it is isogenous (over $\overline \Q$) to its Galois conjugate. Throughout the paper, when saying that curves are isogenous, without mentioning over which field, we will always mean over $\overline \Q$. If the cyclic isogeny $E\rightarrow E^\sigma$ is of degree $d$, we say that $E$ is a $\Q$-curve of degree $d$. By $C_m$ we will denote a cyclic subgroup of order $m$. We allow $\Q$-curves to have complex multiplication. In particular, all elliptic curves with CM are $\Q$-curves.% (see e.g. \cite[Chapter 2]{silverman-advanced}).

We first consider the simpler case of the moduli interpretation of the rational points on $X_0^+(n):=X_0(n)/w_n$. Let $x\in Y_0(n)(\overline \Q)$ represent the pair $(E,C_n)$; then $w_n(x)=(E/C_n, E[n]/C_n)$. Here, quotienting out by $C_n$ means mapping by the $n$-isogeny $\mu$ such that $\ker \mu=C_n$; the subgroup $E[n]/C_n$ will be the kernel of the dual isogeny $\hat \mu$.  Then a rational point on  $X_0^+(n)$ represents a pair 
$$\left\{(E, C_n), (E/C_n, E[n]/C_n)\right\}$$
fixed by $\Gal_\Q$. If this action is trivial, then both $E$ and $E/C_n$ are defined over $\Q$. If the action is non-trivial, it follows that $E$ and $E/C_n$ are defined over a quadratic field, and since $E$ is $n$-isogenous (over $\overline \Q$) to its Galois conjugate, it further follows that $E$ is a  $\Q$-curve of degree $n$ . %

Consider now the general case. Let $n$ be a positive integer and factor $n=dm$ with $(d,m)=1$. Let $w_d$ be the Atkin-Lehner involution sending a point $x\in Y_0(n)(\overline{\Q})$, where $x$ corresponds to $(E,C_d,C_{m})$, to the point $w_d(x)$, corresponding to $(E/C_d,E[d]/C_d,(C_{m}+C_d)/C_d)$. Here, quotienting out by $C_d$ means mapping by the $d$-isogeny $\mu$ such that $\ker \mu=C_d$, i.e. $w_d(x)=(\mu(E), \mu(E[d]),\mu(C_m))$.  Note that all the fixed points of $w_d$ correspond to CM elliptic curves.
%In the case of $d=n$, the curve $X_0(n)/w_n$ is denoted by $X_0^+(n)$. 
Non-cuspidal $\Q$-rational points on $X_0(n)/w_d$ correspond to pairs
$$\left \{(E, C_d, C_m), (E/C_d,E[d]/C_d,(C_{m}+C_d)/C_d) \right\}$$
which are $\Gal_\Q$-invariant, meaning that either the point $(E, C_d, C_m)$ is defined over $\Q$ or there exists a quadratic extension $K/\Q$ with $\sigma$ generating $\Gal(K/\Q)$ such that
\begin{equation}(E, C_d, C_m)^\sigma=(E/C_d,E[d]/C_d,(C_{m}+C_d)/C_d),\label{eq:qc}\end{equation}
implying that $E$ is a $\Q$-curve of degree $d$ with the additional property that $\mu(C_{m})=C_{m}^\sigma$. We will say that an elliptic curve $F$ \textit{corresponds} to a point on $X_0(n)/w_d$ if $j(F)=j(E)$ or $j(F)=j(\mu(E))$.

Let $X/\Q$ be a curve. We say a point $x\in X(\Qbar)$ is $\PP^1$-parametrized if $[\Q(x):\Q]=d$ and there exists a degree $d$ morphism $\varphi:X\rightarrow \PP^1$ over $\Q$ of degree $d$ such that $\varphi(x)\in \PP^1(\Q)$. A point $x\in X(\Qbar)$ that is not $\PP^1$-parametrized is called $\PP^1$-isolated. Note that $\PP^1$-isolated points were previously sometimes referred to as \textit{exceptional} points in the literature; see for example \cite{BruinNajman15}. For hyperelliptic curves in the form $y^2=f(x)$, the quadratic points that are $\PP^1$-parametrized are those for which $x\in \Q$, while all the others are $\PP^1$-isolated.

\section{The ``going-down" method} \label{sec:going-down}

In this section we generalize the ``going-down" methods of \cite[Section 2]{NajmanVukorepa23}. When the quadratic points on $X_0(n)$ have been already classified, it can be useful to use this classification to obtain all the quadratic points on $X_0(N)$ for $N=kn$, where $k > 1$ is an integer. There are two instances where this can be useful: 1) when $X_0(n)$ has finitely many quadratic points, which have all been found, and 2) when $X_0(n)$ has infinitely many quadratic points, and all of them have been described, in the sense that all but finitely many are pullbacks of rational points on some quotient, while the remaining finitely many that are not pullbacks have been explicitly listed. 

The case 1) is of course easier: one need only check whether each of the elliptic curves corresponding to a quadratic point on $X_0(n)$ gives rise to a quadratic point on $X_0(N)$. The cases $n=42,44,45$, with $d=2$ in each case, fall into this category. 

In the case 2), with finitely many explicitly known exceptions, the quadratic points on $X_0(n)$ will be pullbacks of rational points along the quotient map $X_0(n)\rightarrow X_0(n)/w_d$ for some divisor $d$ of $n$. Let $m=N/d.$ The results of \cite[Propositions 2.2 and 2.3]{NajmanVukorepa23} required $m$ to be wither $4$ or a prime. This will not be sufficient for our purposes (and in particular for $N=96$). We will need the following generalization of \cite[Propositions 2.2 and 2.3]{NajmanVukorepa23}.

%; see also \cite[Section 3.2]{AKNOV24}.

%A direct generalisation of 2.2. and 2.3. would be the following:

\begin{proposition} \label{prop:2.1}
    Let $m$ and $d$ be coprime positive integers. Let $E/\overline \Q$ be a non-CM $\Q$-curve of degree $d$ defined over a quadratic field $K$ having in addition an $m$-isogeny defined over $K$ with codomain $E_1$. Then for some divisor $k$ of $m$, $E$ corresponds to a rational point on the curve $X_0\left(\frac{m}{k}d\right)/w_d$, and $E_1$ corresponds to a rational point on the curve $X_0^+(k^2d)$.
\end{proposition}

\begin{proof}
In this proof, for any two isogenous non-CM elliptic curves $E,E'$, we denote by $E\rightarrow E'$ a cyclic isogeny; such an isogeny always exists by \cite[Lemma A.1]{CremonaNajmanQCurve}, and is unique up to sign.

Since $E$ is a $\Q$-curve of degree $d$, it is $d$-isogenous to its Galois conjugate $E^\sigma$. Denote the isogeny $E \to E^\sigma$ by $\mu$, its kernel by $C_d$, and let $C_m$ be the kernel of the $m$-isogeny $E \to E_1$ that exists by assumption. Let $\mu(C_m)\cap C_m^\sigma=k\cdot \mu(C_m)$, where $\sigma$ is the generator of $\Gal(K/\Q)$.

Let $E_2=E/kC_m$ and let $\lambda: E \to E_2$ be the isogeny with kernel $kC_m$. We claim that $E_2$ and $E_2^\sigma$ are $d$-isogenous. Let $E_3=E_2/\lambda(C_d)$ and let $\phi:E_2 \to E_3$ be the isogeny with kernel $\lambda(C_d)$. Then $\phi \circ \lambda: E \to E_3$ and $\lambda^\sigma\circ \mu: E \to E_2^\sigma$ have the same kernel $\diamondop{C_d, kC_m}$, so they are equal up to sign. Thus, $E_3$ and $E_2^\sigma$ are isomorphic, so $E_2^\sigma$ and $E_2$ are $d$-isogenous. We get the following diagram, where all the lines represent cyclic isogenies of the denoted degree.
\[
\begin{tikzpicture}
    \matrix (m) [matrix of math nodes, row sep=3em, column sep=6em, 
                 text height=1.5ex, text depth=0.25ex]
    {
        E & E^\sigma \\
        E_2 & E_2^\sigma \\
        E_1 & E_1^\sigma \\
    };
    
    % Horizontal arrows
    \path[-] (m-1-1) edge node[above] {$d$} (m-1-2);
    \path[-] (m-2-1) edge node[above] {$d$} (m-2-2);
    
    % Vertical arrows
    \path[-] (m-1-1) edge node[left] {$m/k$} (m-2-1);
    \path[-] (m-1-2) edge node[right] {$m/k$} (m-2-2);
    
    \path[-] (m-2-1) edge node[left] {$k$} (m-3-1);
    \path[-] (m-2-2) edge node[right] {$k$} (m-3-2);
    
\end{tikzpicture}
\]
In this diagram $E_2$ is the elliptic curve furthest along the isogeny $E\rightarrow E_1$ which is $d$-isogenous to its Galois conjugate. We see that $$\{(E, C_d, kC_m), (E^\sigma, E[d]/C_d, \mu(kC_m))\}$$ corresponds to a rational point on $X_0\left(\frac{m}{k}d\right)/w_d$.

We claim that the isogeny obtained by composing cyclic isogenies $E_1\to E_2$,  $E_2\to E_2^\sigma$ and $E_2^\sigma\to E_1^\sigma$ is cyclic. This is equivalent to showing that the isogenies from $E_2$ to $E_1^\sigma$ and $E_1$ are independent, i.e.\ that their kernels have trivial intersection. 

Let $C_m'$ be a cyclic subgroup of $E$ of order $m$ such that
$$\ker (E\rightarrow E_1^\sigma)=C_d+C_m'.$$
By what we have proved, $C_m'$ satisfies $kC_m'=kC_m$ and $C_m[k]\cap C_m'[k]=\{0\}.$ 
It follows that 
$$E_2=E/kC_m, \quad \ker (E_2 \rightarrow E_1)=C_m/kC_m, \quad \ker (E_2 \rightarrow E_1^\sigma) =(C_d+C_m')/kC_m, $$
and hence 
$$\ker (E_2 \rightarrow E_1) \cap \ker (E_2 \rightarrow E_1^\sigma)=\{0\}.$$
\begin{comment}
Drugačije:

We claim that the isogeny obtained by composing cyclic isogenies from $E_1$ to $E_2$, from $E_2$ to $E_2^\sigma$ and from $E_2^\sigma$ to $E_1^\sigma$ is cyclic. This is equivalent to showing that the isogenies from $E_2$ to $E_1^\sigma$ and $E_1$ are independent, i.e., that their kernels have trivial intersection. 

Denote by $b$ the isogeny $E_2 \to E_1$, and by $f$ the isogeny $E\to E_1$. 

Suppose that for some $P \in E_2$ we have $b(P)=b^\sigma(\phi(P))=0$. Then $P \in \ker(b)=\lambda(C_m)$ so $P=\lambda(Q)$ for some $Q \in C_m$. 

Furthermore, we have $(b^\sigma \circ \phi \circ \lambda)(Q)=0$, so $(b^\sigma\circ \lambda^\sigma \circ \mu)(Q)=0$ or, in other words, $$(f^\sigma \circ \mu)(Q)=0.$$

This implies $\mu(Q)\in C_m^\sigma$ and since $Q  \in C_m$, we have $\mu(Q)\in \mu(C_m)\cap C_m^\sigma=kC_m^\sigma$, so $(\lambda^\sigma \circ \mu)(Q)=0=\phi(\lambda(Q))=\phi(P)$.

However, $\ker \phi \cap \ker b=\{0\}$ and it follows that $P=0$.
\end{comment}
Hence the isogenies $E_2 \rightarrow E_1$ and $E_2 \rightarrow E_1^\sigma$ are independent and $E_1$ is $k^2d$-isogenous to $E_1^\sigma$, so it corresponds to a point on $X_0^+(k^2d)$.
    
\end{proof}

\begin{remark}
    Note that the curves $E, E^\sigma, E_2$ and $E_2^\sigma$ (together with the appropriate subgroups) in the proof of \Cref{prop:2.1} correspond to a rational point on $X_0\left(\frac m k d\right)/\diamondop{w_d, w_{\frac m k}}$ which lifts to a rational point on $X_0\left(\frac m k d\right)/w_d$.
\end{remark}

We immediately get the following corollary, which we will later often use:
\begin{corollary}\label{cor:main}
    Let $n$ and $d$ be such that $X_0(n)$ is hyperelliptic and $w_d$ is the hyperelliptic involution on $X_0(n)$. Let $N=m_1m_2d$, where $m_2d=n$ and $m_1$ is coprime to $d$. Then any quadratic point on $X_0(N)$ that maps to a $\PP^1$-parametrized point on $X_0(n)$ induces a $\Q$-rational non-cuspidal point on both $X_0(k^2d)^+$ and $X_0(\frac {N} k)/w_d$ for some divisor $k$ of $m_1$.
\end{corollary}

\begin{proof}
    Let $E/K$ be the curve corresponding to $x\in X_0(N)(K)$, where $K$ is a quadratic field. Then $E$ also corresponds to a point on $X_0(n)$ which is, by assumption, $\PP^1$-parametrized. Hence, by \cite[Theorem 12]{BruinNajman15} $E$ is a $\Q$-curve, and from the arguments of \cite{BruinNajman15}, it can easily be seen that it is a $\Q$-curve of degree $d$. The result now follows from  \Cref{prop:2.1}.
\end{proof}

%TODO: $m_2$ could have been included in the 1st case?

\begin{comment}
If the setup is such that some prime $p>2$ divides both $\frac{m}{k}$ and $k$, then the curve $E_2$ from the diagram has three $p$-isogenies over $K$. But then all of its isogenies are defined over $K$, and the mod $p$ Galois representation consists only of scalar matrices. Then the determinant is nonsurjective, so there is a nontrivial intersection between $K$ and $\Q(\zeta_p)$. This means that $K$ is the unique quadratic subfield of $\Q(\zeta_p)$, i.e. $K=\Q(\sqrt{\pm p})$, with the sign being $+1$ if $p\equiv 1 \pmod 4$ and $-1$ if $p \equiv -1 \pmod 4$.

\begin{proposition}[{\cite[Proposition 2]{MerelStein01}}]
    Let $p\equiv 1 \pmod 4$. There does not exist an elliptic curve over $\Q(\sqrt{p})$ with all of its subgroups of order $p$ defined over $\Q(\sqrt{p})$.
\end{proposition}

This implies that in the notation \Cref{cor:main} the case where $m/k$ and $k$ are both divisible by $p$ if $p\equiv 1 \pmod 4$ can be excluded. 
\end{comment}

Let $J_0^-(p)$ be the quotient $J_0(p)/(1+w_p)J_0(p)$ of $J_0(p)$. The following result will turn out to be very useful to us.

\begin{theorem}[Momose {\cite[Theorem 0.1]{Momose1987}}] \label{thm:momose}
    Let $N$ be a composite positive integer. If $N$ has a prime divisor $p$ which satisfies the following conditions (i) and (ii), then there are no non-cuspidal non-CM points in $X_0^+(N)(\Q)$:
    \begin{itemize}
        \item[(i)] $p\geq 17$ or $p=11$.
        \item[(ii)] $p\neq 37$ and $\#J_0^-(p)(\Q)<\infty.$
    \end{itemize}
\end{theorem}
\begin{comment}
\subsection{Helpful results}

In case 1) the quadratic points correspond to points on either $X_0(dm^2)^+$ or $X_0(dm)/w_d$. $X_0(N)^+$ has no non-CM points if it is divisible by a $p$ with $J_0(p)$ of dimension $>1$ with $J_0^-(p)$ finite. This is a consequence of \cite[Theorem 0.1]{Momose87}. See also \cite[Proposition 2.11]{Momose87} for some cases where the same thing can be proved, but are not covered by the previous result. 

Corollary no non-CM points on $X_0(6^2*11)^+$, $X_0(2^2*41)^+$, $X_0(3^2*29)^+$.
\end{comment}
\section{Results} \label{sec:results}
In this section, we list the results obtained on quadratic points on $X_0(N)$. In \Cref{tab:1} we list the remaining values of $N\leq 100$ such that the quadratic points on $X_0(N)$ have not yet been determined. There are $4$ columns per case. The value $N$ denotes the $X_0(N)$ we are studying, and the value $g$ in the $4$th column denotes the genus $g(X_0(N))$ of $X_0(N)$. The value $n$ denotes the $X_0(n)$ we are ``going down" to in the sense that we have a description of all the quadratic points on $X_0(n)$. If $X_0(n)$ is hyperelliptic with $\rk J_0(\Q)=0$, then $d$ is such that $w_d$ is the hyperelliptic involution on $X_0(n)$; there is a single exception i.e.~the case $n=48$ where the hyperelliptic involution is not an Atkin-Lehner involution (there is a $*$ in the $d$ column in this case). This case is dealt with separately in \Cref{sec:96}. If $X_0(n)$ is not hyperelliptic then $X_0(n)$ has finitely many quadratic points (note that then $g(X_0(n))>2$ in all these cases); there is a $-$ in the $d$ column. 

\begin{table}[H]
    \centering
    \begin{tabular}{c|c|c|c|||c|c|c|c}
     N & n & d & g & N & n& d & g\\
    \hline
    $66$ & $22$   & $11$ & $9$  &

    $70$ & $35$   & $35$ & $9$   \\

    $78$ & $39$   & $39$ & $11$  &

    $82$ & $41$   & $41$ & $9$   \\

    $84$ & $42$   & $-$ & $11$  &

      $86$ & $43$   & $43$ & $10$  \\

    $87$ & $29$   & $29$ & $9$   &

    $88$ & $44$   & $-$ & $9$  \\ 
    
    $90$ & $45$   & $-$ & $11$  &

    $93$ & $31$   & $31$ & $9$   \\

    $96$ & $48$   & $*$ & $9$   &

    $99$ & $33$   & $11$ & $9$  \\
\end{tabular}
    \caption{The remaining $X_0(N)$ for $N\leq 100$}
    \label{tab:1}
\end{table}

Next, we will need information about the modular curves $X_0(n)$ we are going down to. \Cref{tab:2} lists the values $n$ and $d$, where $X_0(n)$ is the modular curve we are ``going down" to, such that $X_0(n)$ is hyperelliptic with $\rk J(\Q)=0$, and $d$ is a divisor of $n$ such that $w_d$ is the hyperelliptic involution of $X_0(n)$; see the tables in \cite{bruin_najman2016}. 

\begin{table}[H]
    \centering
    \begin{tabular}{c|c||c|c||c|c||c|c}
     n & d & n & d & n & d& n & d\\
    \hline
    $22$ & $11$   &

    $23$ & $23$  &

    $26$ & $26$ &

    $28$ & $7$ \\

    $29$ & $29$ &

    $30$ & $15$ &

    $31$ & $31$  &

    $33$ & $11$ \\

    $35$ & $35$  &

    $39$ & $39$ &

    $41$ & $41$  &

    $46$ & $23$ \\

    $47$ & $47$  &

    $50$ & $50$ &

    $59$ & $59$  &

    $71$ & $71$ \\

  \end{tabular}
    \caption{Hyperelliptic $X_0(n)$ with rank 0 Jacobian}
    \label{tab:2}
\end{table}

We will now split the cases by the approach that we use, depending on the properties of $X_0(N)$ and $X_0(n)$.

\subsection{The cases where $X_0(n)$ has finitely many quadratic points} \label{subsec:41}

The easiest cases are when $X_0(n)$ has finitely many quadratic points. In these cases $N=84,88,90$, and $n=42, 44$ or $45$, respectively; the quadratic points on these $X_0(n)$ can be found in \cite{OzmanSiksek19}. For $n=42$ and $44$, all the points are CM points.

\begin{remark} \label{rem:isogeny_count}
We briefly describe how one can show, for any number field $K$ and for $N=qn$ with $q$ a prime power not dividing $n$, that the fact that a non-CM non-cuspidal point $x\in X_0(n)(K)$ does not lift to a point on $X_0(N)(K)$ can sometimes be deduced from the cardinality of the number of elliptic curves with a $n$-isogeny in the isogeny class over $K$ of the elliptic curve corresponding to $x$. Let $N=\prod_{i=1}^kq_i$ be the factorization of $N$ into distinct prime powers $q_i$. Then, for every $1<m=\prod_{i=1}^kq_i^{\alpha_i}$ with $\alpha_i\in  \{0,1\}$, we have an Atkin-Lehner involution $w_m$. Furthermore, non-CM points are not fixed by any $w_m$. Hence, by applying the $w_m$ we obtain $2^k$ points on $X_0(N)$, which in turn map to $2^k$ points on $X_0(n)$ corresponding to $2^k$ different elliptic curves.  It follows that if the number of elliptic curves with an $n$-isogeny in the isogeny class over $K$ of the elliptic curve corresponding to $x$ is less than $2^k$, we can be sure that $x$ does not lift to a $K$-rational point on $X_0(N)$.
\end{remark}

The modular curve $X_0(45)$ has two isogeny classes of quadratic points, each consisting of four non-CM curves with $45$-isogenies. Applying \Cref{rem:isogeny_count}, we immediately see that none of these points lift to quadratic points on $X_0(90)$. Since $90$ has three prime divisors, each non-CM isogeny class on $X_0(90)$ must contain at least eight points.

Having dealt with the non-CM points, it remains to check for all three values of $N=84, 88, 90$  whether any of the quadratic CM points on $X_0(n)$ lift to quadratic points on $X_0(N)$, and if yes, to how many points. 

The results of \cite{CGPS22}, as listed in the GitHub repository \url{https://github.com/fsaia/least-cm-degree}, can be used for showing that there are no CM points on $X_0(N)$ when this is the case. In particular, this is true for the values $N=70,87,90$ and $96$ of those that we consider.

When there exist CM points on $X_0(N)$, one way to count the number of points corresponding to each CM $j$-invariant, without doing computer calculations, is by studying the isogeny graphs of these CM elliptic curves. 

\begin{comment}

\begin{example}\label{ex31}
    There are 2 points, up to conjugation (so $4$ in total) on $X_0(33)$ corresponding to the elliptic curve with $j=8000$ with CM by $\Z[\sqrt{-2}]$. Let
    $$\fp_2^2=2\Z[\sqrt{-2}], \quad \fp_{3,1}\fp_{3,2}=3\Z[\sqrt{-2}], \quad \fp_{11,1}\fp_{11,2}=11\Z[\sqrt{-2}]  $$
be the factorizations of $2,3$ and $11$ in $\Z[\sqrt{-2}]$. Then the 4 points on $X_0(33)$ corresponding to $j=8000$ are multiplication by $\fp_{3,i}\fp_{11,j}$ for $1\leq i,j\leq 2.$ Multiplication by $\fp_2$ is the only $2$-isogeny of elliptic curves with $j=8000$ over $\Q(\sqrt{-2})$ (the full $2$-torsion is defined over $\Q(\sqrt 2)$). Hence, we conclude that all $66$-isogenies on this elliptic curve correspond to multiplication by $\fp_2\fp_{3,i}\fp_{11,j}$ for $1\leq i,j\leq 2.$ Thus, there are $4$ such points, 2 counting up to Galois conjugation, so the multiplicity (i.e. the number of points with this $j$-invariant) is 2.
\end{example}
\end{comment}

%\begin{comment}
\begin{example}\label{ex31}
    Consider the case $N=88$. We have $n=44$. Let $K=\Q(\sqrt{-7})$. The ring of integers is $\OO_K=\Z\left[\frac{1+\sqrt -7}{2} \right]$. Elliptic curves with $j$-invariant $-3375$ have CM by $\OO_K$ and elliptic curves with $j=16581375$ have CM by $\Z[\sqrt {-7}]$.
    Let 
    $$\fp_{2,1}\fp_{2,2}=2\OO_K, \quad \fp_{11,1}\fp_{11,2}=11\OO_K$$
    be factorizations in $\OO_K$. Note that $2$ does not split in $\Z[\sqrt {-7}]$, but that $$8=\left(1+\sqrt{-7}\right)\left(1-\sqrt{-7}\right).$$ There is a $2$-isogeny $f$ from an elliptic curve with $j=-3375$ to an elliptic curve with $j=16581375$, and a dual  $2$-isogeny $\hat f$ in the other direction. The $8$ points on $X_0(44)$ with $j=-3375$ correspond to 
    $$\fp_{2,i}^2\fp_{11,j} \ \text{ and } \ f\circ \fp_{2,i}\fp_{11,j} \ \text{ for } \ 1\leq i,j\leq 2,$$
    while the $8$ points on $j=16581375$ correspond to 
    $$\quad  \fp_{2,i}\fp_{11,j} \circ \hat f \ \text{ for } \ 1\leq i,j\leq 2.$$
It follows that the quadratic points on $X_0(88)$ for $j=-3375$ correspond to $$\fp_{2,i}^3\fp_{11,j} \ \text{ and } \ f\circ \fp_{2,i}^2\fp_{11,j} \ \text{ for } \ 1\leq i,j\leq 2,$$
    while the points on $j=16581375$ correspond to 
    $$\quad  \fp_{2,i}^2\fp_{11,j} \circ \hat f \ \text{ for } \ 1\leq i,j\leq 2, \ \text{ and } \ (1\pm \sqrt{-7})\fp_{11,j} \ \text{ for } \  j=1,2.$$
Hence there are $4$ quadratic points on $X_0(88)$ up to conjugacy ($8$ in total) corresponding to each of $j=-3375$ and $j=16581375$. 
\end{example}

Another possibility (which we always use as a sanity check) is to use the code developed in \cite{AKNOV24} for computing the number of quadratic points on $X_0(N)$ above a fixed (not necessarily CM) $j$-invariant. See \githubbare{86.m}, \githubbare{88.m} and \githubbare{90.m} in our Github repository for details. 

\subsection{The cases where $X_0(n)$ is hyperelliptic and a $w_d$ is the hyperelliptic involution.} \label{sec:hyp}
This case covers the cases $N=66,70,78,82,87,99$ and $n=22, 35, 39,41,$ $ 29, 33$, respectively. It does not cover the case $N=96$, where $X_0(48)$ is hyperelliptic, but the hyperelliptic involution is not Atkin-Lehner. The case $N=93$ (with $n=31$) also falls into this category, but needs an additional argument, which we give in \Cref{sec:93}. Let $d$ be such that $w_d$ is the hyperelliptic involution on $X_0(n)$. By \Cref{cor:main}, since $N/n$ is prime, all the non-CM quadratic $\PP^1$-parametrized points on $X_0(N)$ correspond to rational points on $X_0(N)/w_d$ (if $k=1$) or $X_0\left(\frac{N^2d}{n^2}\right)^+$ (if $k=\frac N n$). For $N=66,78,87,99$ the curve $X_0\left(\frac{N^2d}{n^2}\right)^+$ has no rational non-cuspidal non-CM points by \Cref{thm:momose} and the same is true for $N=70$ by \cite[Proposition 2.1. (a)]{Momose87}. In all the cases except for $N=82,99$, the curve $X_0(N)/w_d$ turns out to be a genus $2$ curve whose Jacobian has rank $0$ over $\Q$; it is trivial to find all the rational points on these curves. In the cases of $N=82,99$, the curve $X_0(N)/w_{d}$ is a genus 3 curve whose Jacobian has rank 1 over $\Q$. We compute all rational points on $X_0(N)/w_{d}$ using the (classical) Chabauty method. We use the implementation from \url{https://github.com/jtuitman/Coleman}, for which the theory is described in \cite{baltuit10}. 

Next we need to explicitly check that the exceptional (or $\PP^1$-isolated) points of $X_0(n)$ do not lift to quadratic points on $X_0(N)$. One way this can sometimes easily be done is by seeing that the number of points over a fixed quadratic field would need to be larger than it is, as we did for $N=90$ in \Cref{subsec:41}. 

Another approach, which always works, to do this for non-CM points is as follows. Let $p:=N/n$ (this is always a prime in our cases) and let $a$ be the power of $p$ dividing $N$, but not $n$. Let $E/K$, for $K$ a quadratic field, be an elliptic curve having the same $j$-invariant as the point on $X_0(n)(K)$ that we want to eliminate. Note that the choice of a a particular quadratic twist of $E$ is irrelevant, as having an $N$-isogeny is a quadratic-twist invariant property. If $E$ had an $a$-isogeny over $K$, then $\Gal(\overline K /K)$ would act on the set of $x$-coordinates of the generators of the kernel of such an isogeny.
We factor the primitive $a$-th division polynomial $\psi'_{a}:=\psi_a/\psi_{\frac ap}$ (where $\psi_n$ is the standard $n$-th division polynomial) of $E$. If $\psi'_{a}$, whose roots are all the $x$-coordinates of generators of cyclic subgroups of order $a$ of $E[a]$, has no factors of degree $\leq \varphi(a)/2$, we conclude that there is no $\Gal(\overline K /K)$-invariant cyclic subgroup of order $a$ of $E$ over $K$, and hence no $a$-isogeny. In the most common case of $a=2$, this amounts to just checking whether $E(K)$ has a 2-torsion point.

Finally, we need to check whether there are any CM points (possibly even coming from $\PP^1$-parametrized points on $X_0(n)$) using the results of \cite{CGPS22}, and if there are, what their multiplicities are. This is done, as explained already in \Cref{subsec:41}, either by looking at their isogeny graphs as in \Cref{ex31}, or by explicitly computing the number of fibers that map to the fixed CM $j$-invariant $j_0$ with respect to the $j$-map $j:X_0(N)\rightarrow X_0(1)$.

\subsection{The case $X_0(86)$}
  The curve $X_0(86)$ is the only case where $X_0(n)$ (in this case $n=43$) is neither hyperelliptic nor does it have finitely many quadratic points. However, the quadratic points on $X_0(43)$ are described in \cite[\S 4.1]{Box19} and it is proved that all but 8 points on $X_0(43)$ are $\Q$-curves of degree $43$. This means that we can apply our methods after checking that the 8 exceptional points do not lift to quadratic points on $X_0(86)$. The curve $X_0(86)/w_{43}$ is a genus 4 curve whose Jacobian has rank 2 over $\Q$. We again apply the Chabauty method as described in \cite{baltuit10} to determine all the rational points. 

   We note that $X_0(86)$ is the only case $N\leq 100$ that we consider here for which the Atkin-Lehner sieve from \cite{AKNOV24} succeeds. The Atkin-Lehner sieve also reduces finding quadratic points on $X_0(86)$ to finding rational points on $X_0(86)/w_{43}$. So in this case, our "going down" method produces the same results as the Atkin-Lehner sieve, but without any computation. We note that the solution for this case was included in the Github repository of \cite{AKNOV24}, although it was not mentioned in the paper. We include the list of quadratic points in \Cref{sec:tables} for the sake of completeness.

\subsection{The case $X_0(93)$} \label{sec:93}
We apply \Cref{cor:main} with $n=31$, $d=31$, $m_2=1$ and $m_1=3$; the curve $X_0(31)$ is hyperelliptic, $w_{31}$ is its hyperelliptic involution and $J_0(31)(\Q)$ is finite (see \Cref{tab:2}). For $k=3$ we obtain the curve $X_0^+(279)$, which has no non-cuspidal non-CM rational points by \Cref{thm:momose} (with $p=31$). For $k=1$ we need to find the rational points on the genus $5$ curve $X_0(93)/w_{31}$.

The curve $X_0(93)/w_{31}$ is a degree $2$ cover of $X_0^*(93)$, which is in turn a genus $2$ curve whose Jacobian has rank $2$ over $\Q$. The rational points on $X_0^*(93)$ were determined in \cite{ACKP22}. To get a quadratic point on $X_0(93)$, it is necessary (but not sufficient) that the $j$-invariants of the points listed in \cite{ACKP22} are defined over $\Q$ or a quadratic field. This immediately rules out the non-CM points. We conclude that the $\PP^1$-parametrized non-CM quadratic points on $X_0(31)$ do not lift to quadratic points on $X_0(93)$.

We now check whether the $\PP^1$-isolated quadratic points on $X_0(31)$ lift to quadratic points on $X_0(93)$. We look at the $\PP^1$-isolated points found in \cite{BruinNajman15} and check that none of them lift to quadratic points on $X_0(93)$ by showing that none of the corresponding elliptic curves have a $3$-isogeny over the quadratic fields they are defined over.

Finally, from \cite{CGPS22} we see that $X_0(93)$ has quadratic CM points corresponding to the $j$-invariants $0$, $54000$, $-12288000$ and $-32768$. An explicit computation determines the number of quadratic points on $X_0(93)$ corresponding to each of these $j$-invariants; we list them in \Cref{table93}. The code verifying all the computations for $X_0(93)$ can be found in \githubbare{93.m}.

\subsection{The case $X_0(96)$} \label{sec:96}

The case of $X_0(96)$ is more complicated. We will need some finer considerations of the going-down method to deal with it.

%\begin{proposition}
%    Let $m$ and $n$ be positive integers, where $n$ is odd, such that there exists a finite set $S$ such that all quadratic points on $X_0(2^mn)$ outside of $S$ preimages of the rational points on the quotient of $X_0(2^mn)$ by some involution $w$, where $w$ is not an Atkin-Lehner involution. 

%    Suppose that for all quadratic points $P\not \in S$ on $X_0(2^mn)$ $D[2^\infty]\nsubseteq C[2^\infty]$, where $C$ and $D$ are defined as follows. We have $P=(E,C)$ where $E/K$ is an elliptic curve over a quadratic field $K$ and $C$ a cyclic group of order $2^mn$ defined over $K$, $\sigma $ the non-trivial automorphism of $K$ and $D$ be the kernel of the isogeny from $E$ to a twist of $E^\sigma$. 

   % Then all the quadratic points on $X_0(2^{n+1}m)$ that do not map to $S$ in $X_0(2^nm)$ under the projection map are CM points. 
%\end{proposition}
%\begin{proof} Assume that $E$ does not have CM. We view all curves in the proof up to $\overline \Q$-isomorphism. We have the following $2$-isogeny diagram, where all lines represent $2$-isogenies, all defined over $K$:\\
%    \begin{tikzcd}
%    E_0 \ar[r,dash] & E_1 \ar[r,dash]  \ar[d, dash] & E_2 \ar[d, dash]  \ar[r, dash] &\ldots  \ar[r, dash] & E_{m-1} \ar[d, dash] \ar[r, dash] & E_m \\
%    & F_1 & F_2 & & F_m &
%    \end{tikzcd}
%    By our assumptions $E_0^\sigma$ is isomorphic to some $F_k'$, where $F_k'$ is isogenous by an odd degree isogeny to a unique $F_k$ (and is not isogenous by an odd degree isogeny to any $E_k$). Moreover, since $E_0^\sigma$
%\end{proof}

\begin{proposition}\label{case96}
    The modular curve $X_0(96)$ has no quadratic points.
\end{proposition}
\begin{proof} We view all curves in the proof up to $\overline \Q$-isomorphism. We start by observing that there are no CM elliptic curves over quadratic fields with $96$-isogenies by the results of \cite{CGPS22}.
    By \cite[Section 3.4]{BruinNajman15}, every non-CM elliptic curve $E$ with a cyclic subgroup $C$ of order 48, both defined over a quadratic field $K$, is $12$-isogenous to its Galois conjugate $E^\sigma$ by an isogeny whose kernel has intersection with $C$ of order $6$.
    
    Assume that $E$ has a 96-isogeny and does not have CM. It follows we have the following $2$-isogeny diagram, where all lines represent $2$-isogenies, all defined over $K$.\\
    \begin{center} \begin{tikzcd}
    E_0 \ar[r,dash] & E_1 \ar[r,dash]  \ar[d, dash] & E_2 \ar[d, dash]  \ar[r, dash] &E_3  \ar[r, dash] \ar[d, dash] & E_{4} \ar[d, dash] \ar[r, dash] & E_{5} \\
    & F_1 & F_2 & F_3& F_4 & 
    \end{tikzcd}\\
    \end{center}
    Each of the elliptic curves $E_i$ and $F_j$ in the diagram also has a $3$-isogenous curve $E_i'$ and $F_j'$, and we have an analogous $2$-isogeny diagram.
    \begin{center} \begin{tikzcd}
    E_0' \ar[r,dash] & E_1' \ar[r,dash]  \ar[d, dash] & E_2' \ar[d, dash]  \ar[r, dash] &E_3'  \ar[r, dash] \ar[d, dash] & E_{4}' \ar[d, dash] \ar[r, dash] & E_{5}' \\
    & F_1' & F_2' & F_3' & F_4' & 
    \end{tikzcd}\\
    \end{center}
    By \cite[Section 3.4]{BruinNajman15}, since $E_0$ is $48$-isogenous to $E_4'$, it follows that $E_0^\sigma$ is $12$-isogenous to $E_0$ by an isogeny whose kernel has intersection with $\ker(E_0\to E_4')$ of order $6$. It follows that $E_0^\sigma$ is isomorphic to $F_1'$. 

    However, applying analogous reasoning to $E_1$ and the isogeny $E_1\rightarrow E_5'$, it follows that $E_1^\sigma$ is isomorphic to $F_2'$.

    Since $E_0$ and $E_1$ are $2$-isogenous, it follows that $E_0^\sigma$ and $E_1^\sigma$ are $2$-isogenous. Hence, $F_1$ and $F_2$ are $2$-isogenous. It follows that the $2$-isogeny graph contains a cycle formed by $E_1, E_2, F_2, F_1$. However, this is a contradiction with the fact that the $\ell$-primary subgraph (consisting of $\ell^n$-isogenous elliptic curves) of an isogeny graph of non-CM elliptic curves is always a (regular) tree; see e.g. \cite[\S A.1.5]{CremonaNajmanQCurve} or \cite[\S 2]{Elkies2004}.
    %This implies that these curves have CM, since by composing the isogenies along the cycle one obtains an endomorphism which is not of the form $[n]$ for any $n$. Since the property of having CM is isogeny-invariant, all elliptic curves 
    %in the isogeny class have CM, contradicting our assumption. 
    \end{proof}

We can adapt the proof of this proposition to obtain a more general version.

\begin{proposition}
    Let $K$ be a quadratic number field. Let $m,n$ and $k$ be positive integers such that $2^k$ divides $n$. Suppose that any elliptic curve $E/K$ with a $2^{m-1}n$-isogeny defined over $K$ is $n$-isogenous over $K$ to its Galois conjugate, with the intersection of kernels of the $n$-isogeny and the $2^{m-1}n$-isogeny of size $n/2^k$. 
    Then any elliptic curve $E/K$ with a $2^mn$-isogeny over $K$ is a CM curve.
\end{proposition}
\begin{remark} \label{rem:recover96}
    Taking $m=3$, $n=12$ and $k=1$, we recover the result of  \Cref{case96}.
\end{remark}

\begin{proof}
    Let $n=2^r\cdot s$, with $s$ odd. 
    Suppose that there is a non-CM elliptic curve $E_0/K$ with a $2^m\cdot n$-isogeny. Then we have the following $2$-isogeny diagram. 
\begin{center}
    \begin{tikzcd}
    E_0 \ar[r,dash] & E_1 \ar[r,dash]  \ar[d, dash] & E_2 \ar[d, dash]  \ar[r, dash] &\ldots  \ar[r, dash] & E_{m+r-1} \ar[d, dash] \ar[r, dash] & E_{m+r} \\
    & F_1 & F_2 & & F_{m+r-1} &
    \end{tikzcd}
\end{center}
Each of the elliptic curves $E_i$ and $F_j$ in the diagram also has an $s$-isogenous curve $E_i'$ and $F_j'$, and we have another $2$-isogeny diagram.
\begin{center}
    \begin{tikzcd}
    E_0' \ar[r,dash] & E_1' \ar[r,dash]  \ar[d, dash] & E_2' \ar[d, dash]  \ar[r, dash] &\ldots  \ar[r, dash] & E_{m+r-1}' \ar[d, dash] \ar[r, dash] & E_{m+r}' \\
    & F_1' & F_2' & & F_{m+r-1}' &
    \end{tikzcd}
\end{center}
By the assumption, $E_0$ is isogenous to $E_0^\sigma$ by an isogeny of degree $n$ whose kernel intersected with the kernel of the $2^{m-1}n$-isogeny $E_0 \to E_{m+r-1}'$ has cardinality $n/2^k$.

From here, it follows that $E_0^\sigma$ is isomorphic to $F_{m+r-k-1}'$. Using the same reasoning, $E_1^\sigma$ is isomorphic to $F_{m+r-k}'$. But then $F_{m+r-k-1}$ and $F_{m+r-k}$ are $2$-isogenous and the $2$-isogeny graph contains a cycle. This implies that all elliptic curves in the isogeny class have CM by the same arguments as in the proof of \Cref{case96}. Hence, every elliptic curve $E/K$ with a $2^mn$-isogeny is a CM curve. 

\end{proof}

\section*{Declarations}
\subsection*{Data availability}
All the data generated in this paper is presented in the Tables in \Cref{sec:tables}. Apart from this, no other data has been generated.

\subsection*{Competing interests} The authors declare no competing interests.

\section{Tables} \label{sec:tables}
We will list the results and all the quadratic points below for each value of $N$ considered in the article. We list quadratic points up to conjugation, i.e., we list only one point per Galois conjugacy class. In the tables below, $n$ is a divisor of $N$ for which we have a description of all quadratic points on $X_0(n)$ and to which we are ``going down" to,  and $d$ is the divisor of $n$ such that $w_d$ is the hyperelliptic involution on $X_0(n)$. The values $k$ that appear in the tables are the values of $k$ that come out of \Cref{prop:2.1}; for each $k$ we obtain a quotient of a modular curve for which we need to determine the rational points. The multiplicity denotes how many points on $X_0(N)$ are there above a given $j$-invariant, again counting up to conjugation.

\begin{table}[H]
    \caption{All non-cuspidal quadratic points on \boldmath $X_0(66)$}
    \label{table66}
    \begin{flalign*}
	& \text{Genus: } 9 \\
        &  n=22 \\
        &  d=11 \\
        &  k=1: \ X_0(66)/w_{11} \text{ genus } 2, \text{ rank } 0\\
        & k=3:\ X_0^+(99) \text{ by \Cref{thm:momose}}.
    \end{flalign*}
    {\small
    \begin{tabular}{ccccc}
	\toprule
	Point & Field & $j$-invariant & CM & multiplicity\\ [0.3ex]
	\midrule
        $P_1$ & $\Q\left(\sqrt{-2}\right)$  & $8000$ & $-8$ & $2$ \\ [0.5ex]
        \bottomrule
  \end{tabular}
  } \\
\end{table}

\begin{table}[H]
    \caption{All non-cuspidal quadratic points on \boldmath $X_0(70)$}
    \label{table70}
    \begin{flalign*}
	& \text{Genus: } 9 \\
        &  n=35 \\
        &  d=35 \\
        &  k=1: \ X_0(70)/w_{35} \text{ genus } 2, \text{ rank } 0\\
        & k=2:\ X_0^+(140) \text{ by \cite[Proposition 2.1. (a)]{Momose87} } .\\
        & \textbf{No quadratic non-cuspidal points}
        \end{flalign*}
\end{table}

\begin{table}[H]
    \caption{All non-cuspidal quadratic points on \boldmath $X_0(78)$}
    \label{table78}
    \begin{flalign*}
	& \text{Genus: } 11 \\
        &  n=39 \\
        &  d=39 \\
        &  k=1: \ X_0(78)/w_{39} \text{ genus } 2, \text{ rank } 0\\
        & k=2:\ X_0^+(156) \text{ by \Cref{thm:momose}}.
    \end{flalign*}
    {\small
    \begin{tabular}{ccccc}
	\toprule
	Point & Field & $j$-invariant & CM & multiplicity\\ [0.3ex]
	\midrule
        $P_1$ & $\Q\left(\sqrt{-3}\right)$  & $0$ & $-3$ & $1 $\\ [0.5ex]
        $P_2$ & $\Q\left(\sqrt{-3}\right)$  & $54000$ & $-12$ & $1 $\\
        \bottomrule
  \end{tabular}
  } \\
\end{table}

\begin{table}[H]
    \caption{All non-cuspidal quadratic points on \boldmath $X_0(82)$}
    \label{table82}
    \begin{flalign*}
	& \text{Genus: } 9 \\
        &  n=41 \\
        &  d=41 \\
        &  k=1: \ X_0(82)/w_{41} \text{ genus } 3, \text{ rank } 1\\
        & k=2:\ X_0^+(164) \text{ by \Cref{thm:momose}}.
    \end{flalign*}
    {\small
    \begin{tabular}{ccccc}
	\toprule
	Point & Field & $j$-invariant & CM & multiplicity \\ [0.3ex]
	\midrule
        $P_1$ & $\Q\left(\sqrt{-1}\right)$  & $1728$ & $-4$ & $2 $\\ [0.5ex]
        $P_2$ & $\Q\left(\sqrt{-1}\right)$  & $287496$ & $-16$ & $ 1$ \\ [0.5ex]
        $P_3$ & $\Q\left(\sqrt{-2}\right)$  & $8000$ & $-8$ & $1$ \\
        \bottomrule
  \end{tabular}
  } \\
\end{table}

\begin{table}[H]
    \caption{All non-cuspidal quadratic points on \boldmath $X_0(84)$}
    \label{table84}
    \begin{flalign*}
	& \text{Genus: } 11 \\
        & X_0(42) \text { has finitely many quadratic points} \\
    \end{flalign*}
    {\small
    \begin{tabular}{ccccc}
	\toprule
	Point & Field & $j$-invariant & CM & multiplicity \\ [0.3ex]
	\midrule
        $P_1$ & $\Q\left(\sqrt{-3}\right)$  & $54000$ & $-12$ & $2$\\
        \bottomrule
  \end{tabular}
  } \\
\end{table}

\begin{table}[H]
    \caption{All non-cuspidal quadratic points on \boldmath $X_0(86)$}
    \label{table86}
    \begin{flalign*}
	& \text{Genus: } 10 \\
        &  n=43 \\
        &  d=43 \\
        &  k=1: \ X_0(86)/w_{43} \text{ genus } 4, \text{ rank } 2\\
        & k=3:\ X_0^+(172) \text{ by \Cref{thm:momose}}.\\
    \end{flalign*}
    {\small
    \begin{tabular}{ccccc}
	\toprule
	Point & Field & $j$-invariant & CM & multiplicity \\ [0.3ex]
	\midrule
        $P_1$ & $\Q(\sqrt{-3})$  & $0$ & $-3$ & $1$\\
        [0.5ex]
        $P_2$ & $\Q(\sqrt{-3})$  & $54000$ & $-12$ & $1$\\
        [0.5ex]
        $P_3$ & $\Q(\sqrt{-7})$  & $-3375$ & $-7$ & $3$\\
        [0.5ex]
        $P_4$ & $\Q(\sqrt{-7})$  & $16581375$ & $-28$ & $1$\\
        [0.5ex]
        $P_5$ & $\Q(\sqrt{-2})$  & $8000$ & $-8$ & $1$\\
        \bottomrule
  \end{tabular}
  } \\
\end{table}

%TODO: rational points!

\begin{table}[H]
    \caption{All non-cuspidal quadratic points on \boldmath $X_0(87)$}
    \label{table87}
    \begin{flalign*}
	& \text{Genus: } 9 \\
        &  n=29 \\
        &  d=29 \\
        &  k=1: \ X_0(87)/w_{29} \text{ genus } 2, \text{ rank } 0\\
        & k=3:\ X_0^+(261) \text{ by \Cref{thm:momose}}.\\
        & \textbf{No quadratic non-cuspidal points}
        \end{flalign*}
\end{table}

\begin{table}[H]
    \caption{All non-cuspidal quadratic points on \boldmath $X_0(88)$}
    \label{table88}
    \begin{flalign*}
	& \text{Genus: } 9 \\
        & X_0(44) \text { has finitely many quadratic points} \\
    \end{flalign*}
    {\small
    \begin{tabular}{ccccc}
	\toprule
	Point & Field & $j$-invariant & CM & multiplicity\\ [0.3ex]
	\midrule
        $P_1$ & $\Q\left(\sqrt{-7}\right)$  & $-3375$ & $-7$ & $4$\\
        [0.5ex]
        $P_2$ & $\Q\left(\sqrt{-7}\right)$  & $16581375$ & $-28$ & $4$\\
        \bottomrule
  \end{tabular}
  } \\
\end{table}

\begin{table}[H]
    \caption{All non-cuspidal quadratic points on \boldmath $X_0(90)$}
    \label{table90}
    \begin{flalign*}
	& \text{Genus: } 11 \\
        & X_0(45) \text { has finitely many quadratic points}\\
        & \textbf{No quadratic non-cuspidal points}
    \end{flalign*}
\end{table}

\begin{table}[H]
    \caption{All non-cuspidal quadratic points on \boldmath $X_0(93)$}
    \label{table93}
    \begin{flalign*}
	& \text{Genus: } 9 \\
        &  n=31 \\
        &  d=31 \\
        &  k=1: \ X_0(93)/w_{31} \text{ genus } 5\\
        & k=3:\ X_0^+(279) \text{ by \Cref{thm:momose}}.\\
    \end{flalign*}
    {\small
    \begin{tabular}{ccccc}
	\toprule
	Point & Field & $j$-invariant & CM & multiplicity \\ [0.3ex]
	\midrule
        $P_1$ & $\Q(\sqrt{-3})$  & $0$ & $-3$ & $2$\\
        [0.5ex]
        $P_2$ & $\Q(\sqrt{-3})$  & $54000$ & $-12$ & $1$\\
        [0.5ex]
        $P_3$ & $\Q(\sqrt{-3})$  & $-12288000$ & $-27$ & $1$\\
        [0.5ex]
        $P_4$ & $\Q\left(\sqrt{-11}\right)$  & $-32768$ & $-11$ & $2$\\
        \bottomrule
  \end{tabular}
  } \\
\end{table}
\begin{table}[H]
    \caption{All non-cuspidal quadratic points on \boldmath $X_0(96)$}
    \label{table96}
    \begin{flalign*}
	& \text{Genus: } 9 \\
        & \text {See \Cref{sec:96}}\\
        & \textbf{No quadratic non-cuspidal points}
    \end{flalign*}

\end{table}

\begin{table}[H]
    \caption{All non-cuspidal quadratic points on \boldmath $X_0(99)$}
    \label{table99}
    \begin{flalign*}
	& \text{Genus: } 9 \\
        &  n=33 \\
        &  d=11 \\
        &  k=1: \ X_0(99)/w_{11} \text{ genus } 3, \text{ rank } 1\\
        & k=3:\ X_0^+(99) \text{ by \Cref{thm:momose}}.(\Q).\\
        \end{flalign*}
        {\small
    \begin{tabular}{ccccc}
	\toprule
	Point & Field & $j$-invariant & CM & multiplicity\\ [0.3ex]
	\midrule
        $P_1$ & $\Q\left(\sqrt{-2}\right)$  & $8000$ & $-8$ & $2$\\
        [0.5ex]
        $P_2$ & $\Q\left(\sqrt{-11}\right)$  & $-32768$ & $-11$ & $1$\\
        [0.5ex]
        $P_3$ & $\Q\left(\sqrt{33}\right)$  & $ 3274057859072\sqrt{33} - 18808030478336$ & $-99$ & $1$\\
        \bottomrule
  \end{tabular}
  } \\
\end{table}

\newpage

\bibliographystyle{siam}
\bibliography{bibliography1}
\end{document}